\documentclass{amsart}
\usepackage{amssymb}
\usepackage{graphicx}
\usepackage{epstopdf}
\usepackage{pdfsync}

\numberwithin{equation}{section}

\theoremstyle{plain} 
\newtheorem{thm}[equation]{Theorem}

\newtheorem{lem}[equation]{Lemma}
\newtheorem{prop}[equation]{Proposition}
\newtheorem{claim}[equation]{Claim}
\newtheorem{conjecture}{Conjecture}

\theoremstyle{definition}
\newtheorem{defn}[equation]{Definition}

\theoremstyle{remark}

\graphicspath{{PSF_figures/}}
\title[Tunnel-number-one knots with Seifert-fibered Dehn surgeries]{Hyperbolic tunnel-number-one knots with Seifert-fibered Dehn surgeries}

\author{Sungmo Kang}

\email{skang4450@chonnam.ac.kr}

\begin{document}







\begin{abstract}
Suppose $\alpha$ and $R$ are disjoint simple closed curves in the boundary of a genus two handlebody $H$ such that $H[R]$ embeds in $S^3$ as the exterior of a hyperbolic knot $k$(thus, $k$ is a tunnel-number-one knot), and $\alpha$ is Seifert in $H$(i.e., a 2-handle addition $H[\alpha]$ is a Seifert-fibered space) and not the meridian of $H[R]$. Then for a slope $\gamma$ of $k$ represented by $\alpha$, $\gamma$-Dehn surgery $k(\gamma)$ is a Seifert-fibered space. Such a construction of Seifert-fibered Dehn surgeries generalizes that of Seifert-fibered Dehn surgeries arising from primtive/Seifert positions of a knot, which was introduced in \cite{D03}.

In this paper, we show that there exists a meridional curve $M$ of $k$ (or $H[R]$) in $\partial H$ such that $\alpha$ intersects $M$ transversely in exactly one point. It follows that such a construction of a Seifert-fibered Dehn surgery $k(\gamma)$ can arise from a primtive/Seifert position of $k$ with $\gamma$ its surface-slope.
This result supports partially the two conjectures: (1) any Seifert-fibered surgery on a hyperbolic knot in $S^3$ is integral, and
(2) any Seifert-fibered surgery on a hyperbolic tunnel-number-one knot
arises from a primitive/Seifert position whose surface slope corresponds to the surgery slope.
\end{abstract}

\maketitle


\section{Introduction}
\label{Introduction and main results}

A primitive/Seifert knot $k$, which was introduced by Dean \cite{D03}, is represented by
a simple closed curve $\alpha$ lying a genus two Heegaard surface $\Sigma$ of $S^3$ bounding handlebodies $H$ and $H'$ such that $\alpha$ is primitive in one handlebody, say $H'$, and is Seifert in $H$, that is to say, a 2-handle addition $H'[\alpha]$ is a solid torus and $H[\alpha]$ is a Seifert-fibered space and not a solid torus. Such a pair $(\alpha, \Sigma)$ is called a primitive/Seifert position of $k$. Note that a knot may have more than one primitive/Seifert position. Also note that since $H$ is a genus two handlebody, the Seifert condition of $\alpha$ in $H$ indicates that $H[\alpha]$ is either a Seifert-fibered space over the disk with at most two exceptional fibers or a Seifert-fibered space over the M\"{o}bius band with at most one exceptional fiber. The curve $\alpha$ in the former(the latter, resp.) is said to be Seifert-d(Seifert-m, resp.).

To perform Dehn surgeries on $k$, we consider a surface-slope $\gamma$, which is defined to be an isotopy class of $\partial N(k)\cap \Sigma$, where $N(k)$ is a tubular neighborhood of $k$ in $S^3$. Note that $\alpha$ is isotopic to a component of $\partial N(k)\cap \Sigma$ in $\Sigma$ and thus $\alpha$ can represent the surface-slope $\gamma$.
Also since $\alpha$ intersects a meridional curve of $k$ once, the surface-slope $\gamma$ is integral. Then Lemma 2.3 of \cite{D03} implies that $\gamma$-Dehn surgery $k(\gamma)$ on $k$ is either a Seifert-fibered space over $S^2$ with at most three exceptional fibers or a Seifert-fibered space over $\mathbb{R}P^2$ with at most two exceptional fibers.
Note that a connected sum of lens spaces may arise as a Dehn surgery $k(\gamma)$ but due to \cite{EM92} it can be excluded if a primitive/Seifert knot $k$ is hyperbolic.

Primitive/Seifert knots have some properties. Since $\alpha$ is primitive in $H'$, there exists a complete set of cutting disks $\{D_M, D_R\}$ of $H'$ such that $\alpha$ intersects the boundary $M$ of $D_M$ once transversely and is disjoint from the boundary $R$ of $D_R$. Note that such a cutting disk $D_R$ is unique up to isotopy in $H'$. Then it follows
that $M$ can be considered as a meridional curve of $k$ and $H[R]$ is homeomorphic to the exterior of $k$ in $S^3$, which indicates that such a knot $k$ is a tunnel-number-one knot in $S^3$ such that the curve $R$ is the boundary of a cocore of the 1-handle regular neighborhood of a tunnel. Therefore, if $k$ is a primitive/Seifert knot, then there exist three simple closed curves $\alpha, R,$ and $M$ in the boundary of a genus two handlebody $H$ satisfying:

\begin{enumerate}
\item $\alpha$ is Seifert in $H$.
\item $R$ is disjoint from $\alpha$ such that $H[R]$ is homeomorphic to the exterior of $k$ implying that $k$ is a tunnel-number-one knot.
\item $M$ is a meridional curve of $k$ such that $M$ is disjoint from $R$ and $M$ intersects $\alpha$ once transversely implying that the surface-slope is integral and $\alpha$ represents $k$.
\end{enumerate}

In this paper, by taking only the conditions (1) and (2) we generalize a construction of Seifert-fibered Dehn surgeries arising from primitive/Seifert knots. We will show that the conditions (1) and (2) imply the condition (3), and thus this generalization constructing Seifert-fibered Dehn surgeries narrows down to the construction of Seifert-fibered Dehn surgeries arising from primitive/Seifert knots.

More explicitly, we suppose $\alpha$ and $R$ are disjoint simple closed curves in the boundary of a genus two handlebody $H$ such that $H[R]$ embeds in $S^3$ as the exterior of a hyperbolic knot $k$, and $\alpha$ is Seifert in $H$ and not the meridian of $H[R]$.
Since $\alpha$ is disjoint from $R$, we can consider $\alpha$ as a curve representing a slope $\gamma$ in $\partial N(k)$ of $k$ in $S^3$. Then note that
since $\alpha$ is Seifert in $H$, it follows that the $\gamma$-Dehn surgery $k(\gamma)$ is either a Seifert-fibered space over $S^2$ with at most three exceptional fibers or
a Seifert-fibered space over $\mathbb{R}P^2$ with at most two exceptional fibers.

The main result of this paper is the following theorems.

\begin{thm}\label{main result}
Suppose $\alpha$ and $R$ are disjoint simple closed curves in the boundary of a genus two handlebody $H$ such that $H[R]$ embeds in $S^3$ as the exterior of a hyperbolic knot $k$, and $\alpha$ is Seifert in $H$ and not the meridian of $H[R]$.
Then there exists a meridional curve $M$ of $k$ (or $H[R]$) in $\partial H$ such that $\alpha$ intersects $M$ transversely in exactly one point.
\end{thm}

As a consequence of Theorem~\ref{main result}, we have the following.

\begin{thm}\label{cor of main result}
Suppose $\alpha$ and $R$ are disjoint simple closed curves in the boundary of a genus two handlebody $H$ such that $H[R]$ embeds in $S^3$ as the exterior of a hyperbolic knot $k$, and $\alpha$ is Seifert in $H$ and not the meridian of $H[R]$, whence for a slope $\gamma$ represented by $\alpha$, $k(\gamma)$ is a Seifert-fibered space. Then $(\alpha, \partial H)$ is
a primitive/Seifert position of $k$ and its surface-slope is $\gamma$ so that the Seifert-fibered Dehn surgery $k(\gamma)$ arises from the primitive/Seifert position $(\alpha, \partial H)$.
\end{thm}

\begin{proof} Let $H'$ be the closure of the complement of $H$ in $S^3$. Since $H[R]$ embeds in $S^3$ as the exterior of $k$, $H'$ is a genus two handlebody such that $R$ bounds a cutting disk of $H'$. Since by Theorem~\ref{main result}, there exists a meridional curve $M$ of $k$ (or $H[R]$) in $\partial H$ such that $\alpha$ intersects $M$ transversely in exactly one point, $\alpha$ is primitive in $H'$. Therefore, $(\alpha, \partial H)$ is a primitive/Seifert position of $k$, its surface-slope is $\gamma$, and the Seifert-fibered Dehn surgery $k(\gamma)$ arises from the primitive/Seifert position $(\alpha, \partial H)$, as desired.
\end{proof}

These results support partially the following conjectures.

\begin{conjecture}\label{first conjecture}
Any Seifert-fibered surgery on a hyperbolic knot in $S^3$ is integral.
\end{conjecture}

\begin{conjecture}\label{second conjecture}
Any Seifert-fiberd surgery on a hyperbolic tunnel-number-one knot
arises from a primitive/Seifert position whose surface slope corresponds to the surgery slope.
\end{conjecture}

Conjecture~\ref{first conjecture} is known to be true for various Seifert-fibered Dehn surgeries on a hyperbolic knot.
Due to the famous result of \cite{CGLS87}, if $k(\gamma)$ is a lens space, then $\gamma$ is integral. Boyer-Zhang~\cite{BZ98} proved that the conjecture is true for toroidal Seifert-fibered surgeries. If a Seifert-fibered surgery $k(\gamma)$ has a projective plane as the base surface, then it contains a Klein bottle, in which
case by Gordon-Leucke~\cite{GL95} $\gamma$ is integral.
Thus the only remaining case is when $\gamma$-Dehn surgery $k(\gamma)$ is a Seifert fibered space over the sphere with three exceptional fibers. Theorem~\ref{cor of main result} gives a partial answer for this case.

Regarding Conjecture~\ref{second conjecture}, there are families of hyperbolic knots admitting Seifert-fibered surgeries which do not arise from primitive/Seifert positions. See \cite{MMM05}, \cite{T07}, \cite{DMM12}, \cite{DMM14}, and \cite{EJMM15}.
All of the knots in \cite{MMM05}, \cite{T07}, and \cite{DMM12} are not strongly invertible. Meanwhile, the knots in \cite{DMM14} and \cite{EJMM15} are strongly invertible but do not have tunnel number one. All of the knots above are not primitive/Seifert knots because any primitive/Seifert knot has tunnel number one and thus are strongly invertible. However, it is still unknown that there are examples of Seifert-fibered surgeries on hyperbolic tunnel-number-one knots in $S^3$ which do not arise from primitve/Seifert positions.

The main idea of proving Theorem~\ref{main result} is to use the main result of \cite{B20}, which is originally introduced in \cite{B93}, saying that a meridian of $H[R]$ can be obtained from $R$ by surgery along a distinguished wave, and the main result of \cite{K20b} claiming that there are two types of R-R diagrams of Seifert-d curves in $H$: rectangular form and non-rectangular form, and there is one type of R-R diagram of Seifert-m curves in $H$.

Some related definitions and properties necessary to prove Theorem~\ref{main result} are provided in Section~\ref{Preliminaries}. In Sections~\ref{Seifert-d rectangular curves} and \ref{Seifert-d non-rectangular curves} we prove Theorem~\ref{main result} when $\alpha$
is Seifert-d with rectangular form and with non-rectangular form respectively. Section~\ref{Seifert-m curves} provides the proof of Theorem~\ref{main result} when $\alpha$ is Seifert-m.\\

\noindent\textbf{Acknowledgement.} This paper is originated from the joint work with John Berge for the project of the classification of hyperbolic primitive/Seifert knots in $S^3$. I would like to express my gratitude to John Berge for his collaboration and support. I would also like to thank Cameron Gordon and John Luecke for their great hospitality while I stayed in the University of Texas at Austin.

\section{Preliminaries}
\label{Preliminaries}

We start with the following lemma, which can be found in \cite{HOT80} or \cite{O79} and shows some possible types of graphs of Heegaard diagrams of simple closed curves in the boundary of a genus two handlebody.

\begin{figure}[tbp]
\centering
\includegraphics[width = 1.0\textwidth]{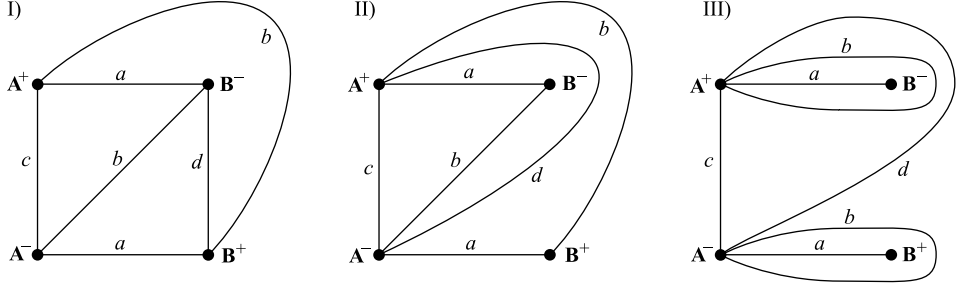}
\caption{The three types of graphs of Heegaard diagrams of simple closed curves on
the boundary of a genus two handlebody $H$ which has cutting disks $D_A$ and $D_B$,
excluding diagrams in which simple closed curves are disjoint from both $\partial D_A$ and $\partial D_B$. }
\label{DPCFig8a-3}
\end{figure}

\begin{lem}
\label{3 types of genus two diagrams}
Let $H$ be a genus two handlebody with a set of cutting disks $\{D_A,D_B\}$ and let $\mathcal{C}$ be a finite set of pairwise disjoint nonparallel simple closed curves on $\partial H$ whose intersections with $\{D_A, D_B\}$ are essential and not both empty. Then, after perhaps relabeling $D_A$ and $D_B$, the Heegaard diagram of $\mathcal{C}$ with respect to $\{D_A, D_B\}$ has the form of one of the three graphs in Figure~\emph{\ref{DPCFig8a-3}}.
\end{lem}

\begin{defn}
[\textbf{cut-vertex}]
If $v$ is a vertex of a connected graph $G$ such that deleting $v$ and
the edges of $G$ meeting $v$ from $G$ disconnects $G$, we say $v$ is a \emph{cut-vertex} of $G$.
\end{defn}

The Heegaard diagram in Figure~\ref{DPCFig8a-3}c either is not connected or has a cut-vertex.

\begin{defn}
[\textbf{Positive Heegaard Diagram}]
A Heegaard diagram is \emph{positive} if the curves of the diagram can be oriented so that all intersections of curves in the diagram are positive. Otherwise, the diagram is \emph{nonpositive}.
\end{defn}

Suppose $R$ is a nonseparating simple closed curve in the boundary of a genus two handlebody $H$ such that $H[R]$ embeds in $S^3$, i.e., $H[R]$ is an exterior of a knot $k$ in $S^3$.
It is shown in \cite{B20}, which is essentially originated from \cite{B93}, that a meridian of $H[R]$ (or $k$) can be obtained from $R$ by surgery along a wave based at $R$. Recall that a wave on the curve $R$ in $\partial H$ is an arc $\omega$ whose endpoints lies on $R$ with the opposite signs. The following is one of the results of \cite{B20}, which shows how to get a meridian of $H[R]$.

\begin{thm}
[\textbf{Waves provide meridians}]
\label{waves provide meridians}
Let $H$ be a genus two handlebody with a set of cutting disks $\{D_A,D_B\}$ and let $R$ be a nonseparating simple closed curve on $\partial H$ such that the Heegaard diagram $\mathbb{D}_R$ of $R$ with respect to $\{D_A,D_B\}$ is connected and has no cut-vertex. Suppose, in addition, that the manifold $H[R]$ embeds in $S^3$. Then $\mathbb{D}_R$ determines a wave $\omega$ based at $R$ such that if $m$ is a boundary component of a regular neighborhood of $ R \cup \omega $ in $\partial H$, with $m$ chosen so that it is not isotopic to $R$, then $m$ represents the meridian of $H[R]$. Furthermore, the wave $\omega$ determined by $R$ can be obtained as follows:

\begin{enumerate}
\item If $\mathbb{D}_R$ is nonpositive, then $\omega$ is a unique vertical wave $\omega_v$ which is isotopic to a subarc of the boundary of one of $D_A$ and $D_B$ with which $R$ has both positive and negative signed intersections.
\item If $\mathbb{D}_R$ is positive, then $\omega$ is a horizontal wave $\omega_h$ such that one endpoint of $\omega_h$ lies on an edge of $\mathbb{D}_R$ connecting vertices $A^+$ and $A^-$, while the other endpoint of $\omega_h$ lies on an edge of $\mathbb{D}_R$ connecting vertices $B^+$ and $B^-$.
\end{enumerate}
\end{thm}

\begin{figure}[tbp]
\centering
\includegraphics[width = 1\textwidth] {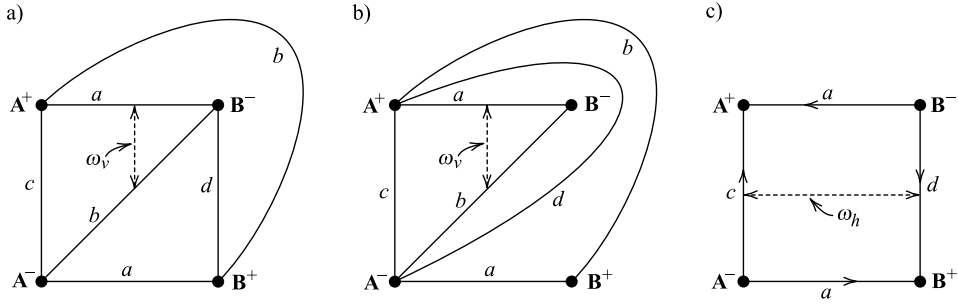}
\caption {A vertical wave $\omega_v$ in a nonpositive Heegaard diagram in a) and b) where $R$ has both positive and negative signed intersections with $D_B$, and a horizontal wave $\omega_h$ in a positive Heegaard diagram in c). They are said to be distinguished
in the sense that they can be used in a surgery on $R$ to obtain a meridian of $H[R]$.}
\label{DPCFig8a-4}
\end{figure}

Figures~\ref{DPCFig8a-4}a and \ref{DPCFig8a-4}b show vertical waves $\omega_v$ when
$R$ has both positive and negative signed intersections with the cutting disk $D_B$ so
that the Heegaard diagram $\mathbb{D}_R$ is nonpositive. Figure~\ref{DPCFig8a-4}c shows a horizon wave $\omega_h$ when $\mathbb{D}_R$ is positive.
Vertical waves and horizontal waves which are used to find a representative of a meridian of $H[R]$ as described in Theorem~\ref{waves provide meridians} are said to be \textit{distinguished}.

Next proposition provides some special type of R-R diagrams of $R$ such that $H[R]$ is nonhyperbolic. For the definition and properties of R-R diagrams, see \cite{K20c}.

\begin{prop}\label{nonhyperbolic R}
Suppose $R$ is a simple closed curve in the boundary of a genus two handlebody $H$ with an R-R diagram of the form shown
in Figure~\emph{\ref{PSFFig3ae}} with $a, b\geq 0$ and $m$, $n$, $s\in \mathbb{Z}$.

If $H[R]$ embeds in $S^3$, then $R$ is either a primitive curve, or a torus or cable knot relator on $H$. Therefore if $k$ is
a knot whose exterior is homeomorphic to $H[R]$, then $k$ is either the unknot, a torus knot or a cable of a torus knot.
\end{prop}
\begin{proof}
This is Theorem 2.4 of \cite{K20a}.
\end{proof}

\begin{figure}[tbp]
\centering
\includegraphics[width = 0.55\textwidth]{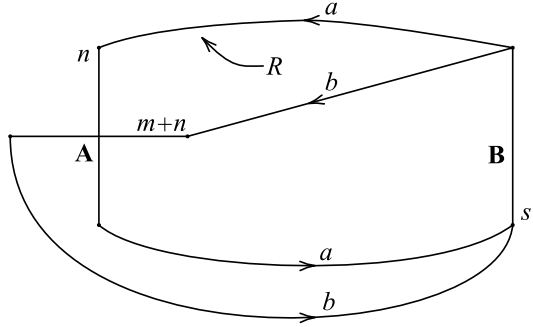}
\caption{R-R diagrams of a simple closed curve $R$ in which $R$ has only one connection on one handle and at most two connections on the other handle. Here $a,b \geq 0$, $\gcd(a,b) = 1$, and $m,n,s \in \mathbb{Z}$.}
\label{PSFFig3ae}
\end{figure}

\section{The case when $\alpha$ is Seifert-$\mathrm{d}$ with rectangular form}
\label{Seifert-d rectangular curves}

The classification theorem of Seifert-d curves in \cite{K20b} says that if $\alpha$ is a Seifert-d curve, then $\alpha$ has an R-R diagram of the forms in Figure~\ref{PSFFig2a1}.
If $\alpha$ has the R-R diagram of the form in Figure~\ref{PSFFig2a1}a (\ref{PSFFig2a1}b, resp.), then we say that $\alpha$ is of a rectangular form (a non-rectangular form, resp.).
In this section and next section we prove Theorem~\ref{main result} for the case when $\alpha$ is of a rectangular form and for the case when $\alpha$ is of a non-rectangular form respectively.

\begin{figure}[tbp]
\centering
\includegraphics[width = 1.0\textwidth]{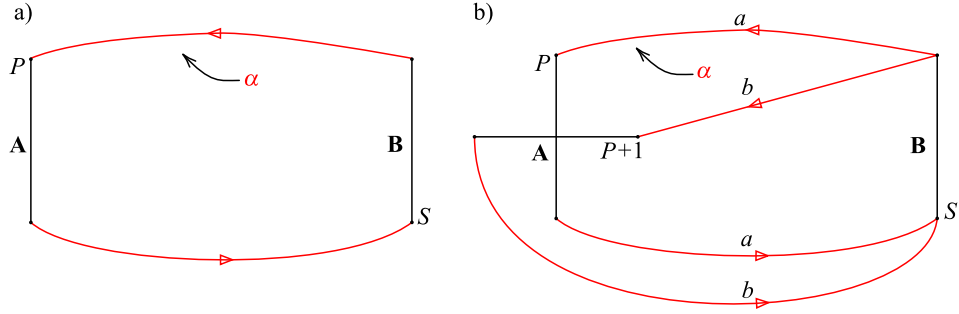}
\caption{If $\alpha$ is a Seifert-d curve in the boundary of a genus two handlebody $H$, then $\alpha$ has an R-R diagram with the form of one of these figures with $P, S > 1$, $a, b > 1$, and $\gcd(a,b) = 1$. If $\alpha$ has an R-R diagram with the form of Figure~\ref{PSFFig2a1}a, we say $\alpha$ has \emph{rectangular form}. Then $H[\alpha]$ is a Seifert-fibered space over $D^2$ with two exceptional fibers of index $P$ and $S$. If $\alpha$ has an R-R diagram with the form of Figure~\ref{PSFFig2a1}b, we say $\alpha$ has \emph{non-rectangular form}. Then $H[\alpha]$ is a Seifert-fibered space over $D^2$ with two exceptional fibers of index $P(a+b)+b$ and $S$.}
\label{PSFFig2a1}
\end{figure}

Suppose $\alpha$ is of a rectangular form, i.e., $\alpha$ has an R-R diagram of the form in Figure~\ref{PSFFig2a1}a.

\begin{prop}
Theorem~\emph{\ref{main result}} holds if $\alpha$ is Seifert-d and has rectangular form.
\label{rectangular prop}
\end{prop}

\begin{proof}
In the R-R diagram of $\alpha$ of a rectangular form in Figure~\ref{PSFFig2a1}a, we add two arbitrary bands of connections in each handle, namely $Q$- and $R$-connections in the $A$-handle, and $U$- and $T$-connections in the $B$-handle as shown in Figure~\ref{noPS}.
Note that $P+R=Q$ and $S+U=T$. Here we overuse the letter $R$ meaning a simple closed curve as well as the label of the connection in the $A$-handle.
However, the confusion will obviously be eliminated in the context.

Now we consider adding a simple closed curve $R$ disjoint from $\alpha$.
We can observe that $R$ cannot have both $P$- and $S$-connections, otherwise the curve $R$ is forced to spiral endlessly and cannot be a simple closed curve.
Therefore up to the symmetry of the R-R diagram of $\alpha$, without loss of generality we may assume that $R$ has no $P$-connections.
There are two cases to consider: (1) $R$ has neither $P$-connections nor $S$-connections and (2) $R$ has no $P$-connections and has $S$-connections.

\smallskip
\textbf{Case (1): The curve $\boldsymbol{R}$ has neither $\boldsymbol{P}$-connections nor $\boldsymbol{S}$-connections.}
\smallskip

\begin{figure}[t]
\includegraphics{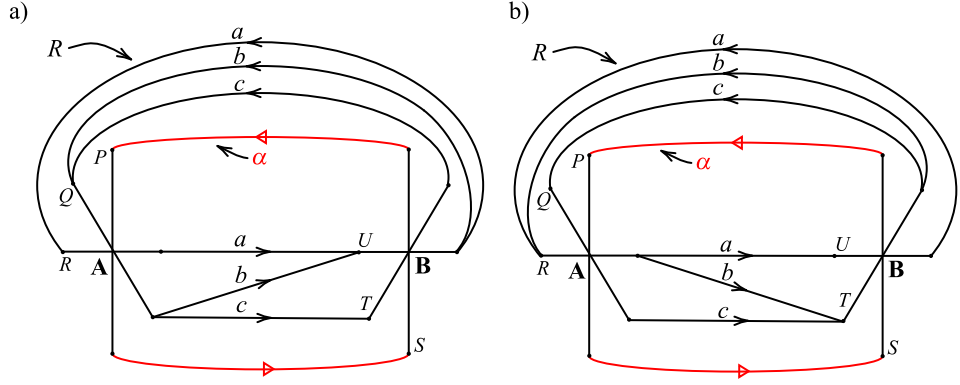}\caption{The R-R diagrams of disjoint nonseparating curves $\alpha$ and $R$ in which $\alpha$ has rectangular form and $R$ has no $P$-connections and no $S$-connections.}\label{noPS}
\end{figure}

A priori there are two possible R-R diagrams of such curves $R$. These appear in Figure~\ref{noPS}. However, examination shows the R-R diagrams of Figure \ref{noPS}a and Figure \ref{noPS}b agree up to homeomorphism and relabeling parameters. So we may suppose $R$ has an R-R diagram of Figure~\ref{noPS}a.

Note that the weights $a, b, c>0$ in Figure~\ref{noPS}, otherwise by Proposition~\ref{nonhyperbolic R}, $H[R]$ is not hyperbolic.

First, suppose $R$ is nonpositive. Since $P,S>1$, none of $R,Q,U,$ and $T$ is $0$. This, when combined with nonpositivity, implies
that the Heegaard diagram of $R$ underlying the R-R diagram is connected and has no cut vertex. Therefore there exists a distinguished vertical wave $\omega_v$
such that by Theorem~\ref{waves provide meridians} a meridian $M$ of $H[R]$ is obtained from $R$ by surgery along $\omega_v$.
It follows immediately from the R-R diagram that the vertical wave $\omega_v$ intersects $\alpha$ transversely at a point.
Therefore a meridian $M$ of $H[R]$ intersects $\alpha$ transversely at a point.

Now we assume that $R$ is positive. The conditions that $a, b, c>0$ and $P,S>1$ implies that $RQ>0$ and $TU>0$, and max$\{|R|, |Q|\}>1$ and max$\{|T|, |U|\}>1$.
Therefore the Heegaard diagram of $R$ is connected and has no cut vertex and by Theorem~\ref{waves provide meridians} there exists a distinguished horizontal wave $\omega_h$
such that a meridian $M$ of $H[R]$ is obtained from $R$ by surgery along $\omega_h$.

\begin{figure}[t]
\includegraphics{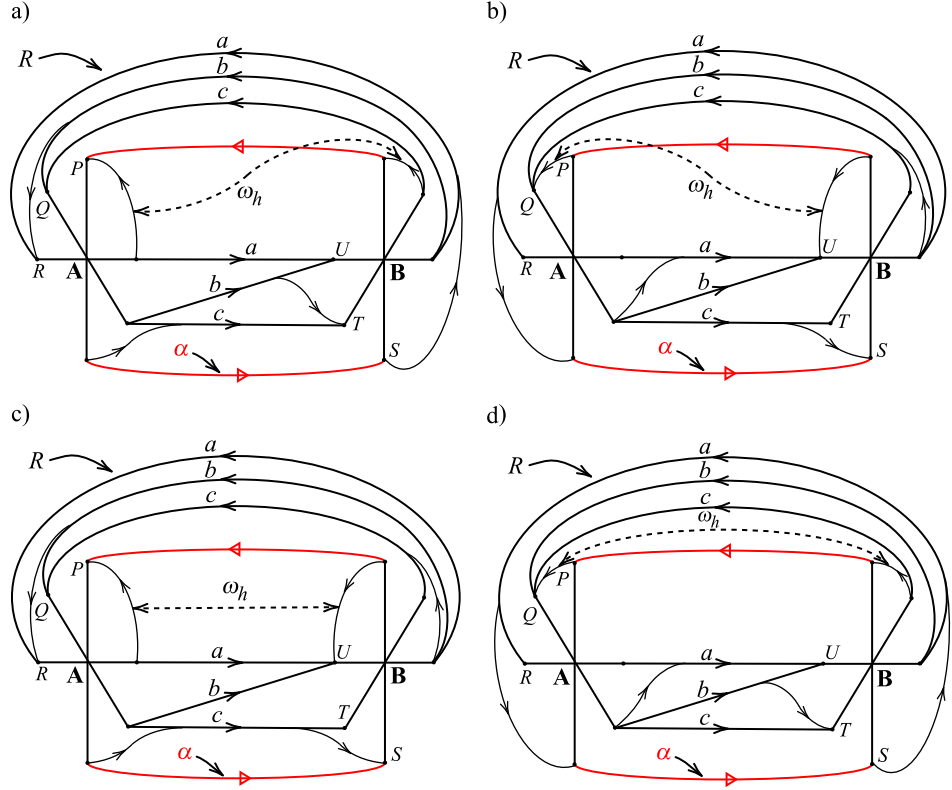}
\caption{R-R diagrams of horizontal waves $\omega_h$ based at the curve $R$ which show how $\omega_h$ depends on the signs of the parameters $R, Q, U,$ and $T$. In \ref{positiverwave}a), $R,Q,P > 0$ and $T,U,-S < 0$. In \ref{positiverwave}b), $-P,R,Q < 0$ and $S,T,U > 0$. In \ref{positiverwave}c),
$R,Q,P > 0$ and $S,T,U > 0$. In \ref{positiverwave}d), $-P,R,Q < 0$ and $T,U,-S < 0$.}
\label{positiverwave}
\end{figure}

Locating the horizontal wave $\omega_h$ in the R-R diagram of $R$
depends on which band of connections of $R$ has maximal labels. See \cite{K20c} for the information on the location of horizontal waves in R-R diagrams. Therefore, it depends on the signs of $R, Q, T,$ and $U$.
There are four cases to consider:
\begin{enumerate}
\item[(a)] $R$, $Q>0$ and $T$, $U<0$;
\item[(b)] $R$, $Q<0$ and $T$, $U>0$;
\item[(c)] $R$, $Q>0$ and $T$, $U>0$;
\item[(d)] $R$, $Q<0$ and $T$, $U<0$.
\end{enumerate}

If $R$, $Q>0$ and $T$, $U<0$, then $Q$ and $U$ are the maximal labels in the $A$- and $B$-handles respectively and $\omega_h$ has an endpoint on a connection in one handle
which borders the band of connections with maximal label. In order to locate
$\omega_h$ in the R-R diagram, we isotope the outermost edge of the $b$ parallel edges entering the $Q$-connection in the $A$-handle and also isotope the outermost edge of the $b$ parallel edges entering the $U$-connection in the $B$-handle as shown
in Figure~\ref{positiverwave}a. Then it follows from \cite{K20c} that $\omega_h$ appears as in Figure~\ref{positiverwave}a. Similarly for the other cases, $\omega_h$
appears as in Figures~\ref{positiverwave}b, \ref{positiverwave}c, and \ref{positiverwave}d.
It follows that in the cases (a) and (b), $|\omega_h \cap \alpha | = 1$ and thus a meridian of $H[R]$ intersects $\alpha$ transversely at a point.
In the cases (c) and (d), $\omega_h \cap \alpha = \varnothing$, which indicates that $\alpha$ is isotopic to a meridian of $H[R]$, a contradiction to the hypothesis of Theorem~\ref{main result}.

\smallskip
\textbf{Case (2): $\boldsymbol{R}$ has $\boldsymbol{S}$-connections but no $\boldsymbol{P}$-connections.}
\smallskip

\begin{figure}[t]
\includegraphics{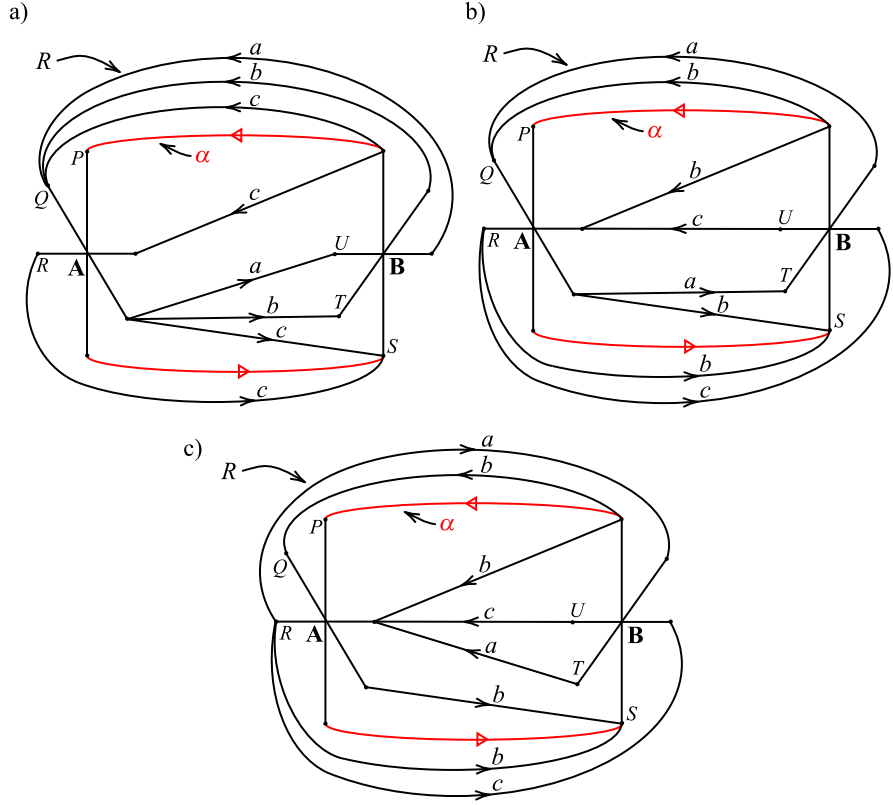}
\caption{R-R diagrams of $\alpha$ and $R$ in which $\alpha$ has rectangular form and $R$ has $S$-connections.}
\label{positivertype}
\end{figure}

There are three possible R-R diagrams of $R$ as shown in Figure~\ref{positivertype}. However, using an orientation-reversing homeomorphism of $H$ and thus of the R-R diagram of $(\alpha, R)$, and relabelling the parameters, we observe that the R-R diagram in Figure~\ref{positivertype}c is equivalent to that of Figure~\ref{positivertype}a.
Therefore, we consider the R-R diagrams of the forms in Figures~\ref{positivertype}a and \ref{positivertype}b.
Since $H[R]$ is hyperbolic, it follows by Proposition~\ref{nonhyperbolic R} that $c>0$ in Figure~\ref{positivertype}a, and $b>0$ in Figure~\ref{positivertype}b.

If $R$ is nonpositive, then since $P, S>1$, it follows that the Heegaard diagram of $R$ is connected and has no cut vertex and thus there exists a distinguished vertical wave $\omega_v$ such that a meridian $M$ of $H[R]$ is obtained from $R$ by surgery along $\omega_v$. Since $c>0$($b>0$, resp.) in Figure~\ref{positivertype}a(\ref{positivertype}b, resp.), $\omega_v$ does not
intersect $\alpha$, a contradiction. Therefore, $R$ is positive.

First, suppose $R$ has the R-R diagram of Figure~\ref{positivertype}a. Since $R$ is positive, $Q, T, U>0$ and $R<0$.

\begin{claim}\label{claim1}
In the diagram of $R$ of Figure~\emph{\ref{positivertype}a}, we may assume that $R+Q \neq 0$.
\end{claim}
\begin{proof}
Suppose $R+Q=0$. Since gcd$(|R|, |Q|) = 1$, $R= -1$, $Q=1$ and $P=2$. Then $\alpha = A^2 B^S$ and the Heegaard diagram of $R$ has a cut vertex. Now we use the argument of the hybrid diagram. Its hybrid diagram of $\alpha$ and $R$ corresponding to the R-R diagram of Figure~\ref{positivertype}a is illustrated in Figure~\ref{fig5-2}a.
For the definition and properties of hybrid diagrams, see \cite{K20c}.

\begin{figure}[tbp]
\centering
\includegraphics[width = 1\textwidth]{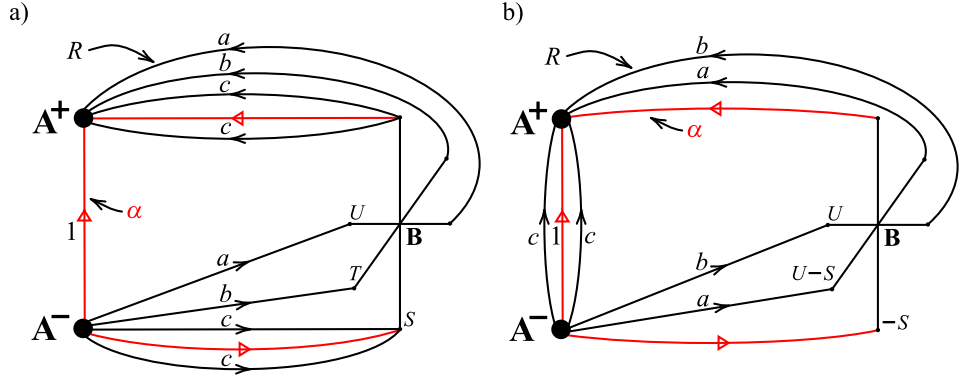}
\caption{The hybrid diagram of Figure~\ref{fig5-2}a corresponds to the R-R diagram of Figure~\ref{positivertype}a. The hybrid diagram in Figure~\ref{fig5-2}b is obtained from the hybrid diagram in
Figure~\ref{fig5-2}a by dragging vertex the $A^-$ of Figure~\ref{fig5-2}a, together with the edges of Figure~\ref{fig5-2}a meeting the vertex $A^-$ of Figure~\ref{fig5-2}a, over the S-connection of the B-handle of Figure~\ref{fig5-2}a. This induces an automorphism of $\pi_1(H)$ which takes $A \mapsto AB^{-S}$.}
\label{fig5-2}
\end{figure}

In its hybrid diagram, we drag the vertex $A^-$
together with the edges of $R$ and $\alpha$ meeting the vertex $A^{-}$ over the $S$-connection on the $B$-handle.
This performance corresponds to a change of cutting disks inducing an automorphism of $\pi_1(H)$ that takes $A \mapsto AB^{-S}$ and leaves $B$ fixed.
With an orientation-reversing homeomorphism of $H$ applied, the resulting hybrid diagram of $\alpha$ and $R$ is depicted in Figure~\ref{fig5-2}b.
It follows from Figure~\ref{fig5-2}b that $R$ has only two bands of connections labelled by $U$ and $U-S$ in the $B$-handle.
For the labels of the bands of connections in the $A$-handle, by chasing the parallel arcs of weight $c$ in the R-R diagram of $R$ of Figure~\ref{positivertype}a,
we observe that the subword $AB^S$ in $R$ appears in the sequence of syllables $\cdots AB^SAB^SA\cdots$.
This implies that after the automorphism taking $A \mapsto AB^{-S}$, $A^2$ does not appear as a single syllable in the word of $R$ in $\pi_1(H)$. On the other hand, $\alpha$ is sent to $A^2B^{-S}$ in $\pi_1(H)$, which is still of a rectangular form. This implies that $\alpha$ and $R$ have no common single syllable, which means that $\alpha$ and $R$ have no common connections. Therefore this case belongs to Case (1) where $R$ has neither $P$-connections nor $S$-connections.
\end{proof}

By Claim~\ref{claim1}, max$\{|Q|, |R|\}>1$. Since $R$ has $S$-connections on the $B$-handle, the Heegaard diagram is connected and has no cut vertex. Therefore it has a distinguished horizontal
wave $\omega_h$ yielding a meridian of $H[R]$. As in the case of (1), locating $\omega_h$ in the R-R diagram of $R$ depends on the sign of $R+Q$ unless $b=0$, in which case $\omega_h$
also depends on the maximal label member of $\{S, U\}$. Figure~\ref{horizontalwave1}, where
the $P$-connection of $\alpha$ is isotoped to the $Q$- and $-R$-connection, shows
$\omega_h$ when $b>0$. In either of the R-R diagrams $\omega_h$ intersects $\alpha$ transversely once. For the case where $b=0$ and $a\neq 0$, we insert $(S-U)$-connection
in the $B$-handle to locate $\omega_h$. Then it is easy to show that in this case $\omega_h$ also intersects $\alpha$ transversely. Note that at least one of $a$ and $b$ must be positive, otherwise $R$ has only two
bands of connections on the $A$-handle and only one band of connections so that by Proposition~\ref{nonhyperbolic R} $H[R]$ is not hyperbolic.

\begin{figure}[t]
\includegraphics[width = 1.0\textwidth]{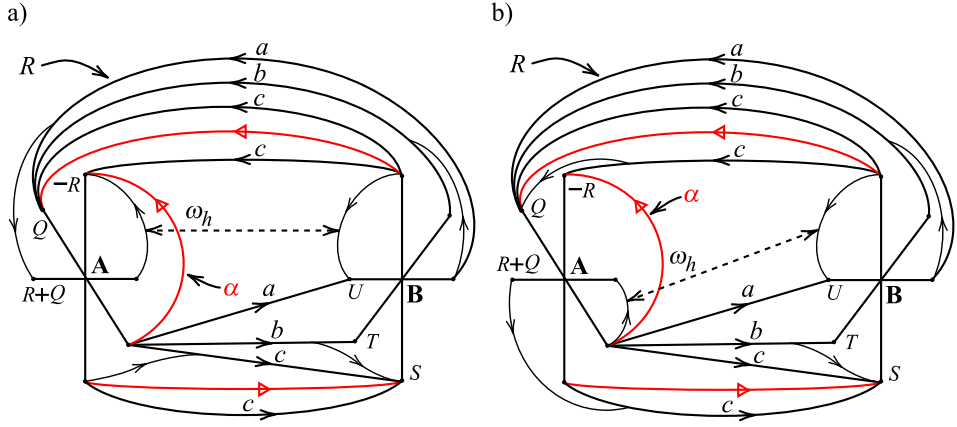}
\caption{Horizontal waves $\omega_h$ in R-R diagrams of $\alpha$ and $R$ when $b > 0$, $R+Q>0$
 in Figure~\ref{horizontalwave1}a, and $R+Q<0$ in Figure \ref{horizontalwave1}b.}
\label{horizontalwave1}
\end{figure}

Now, we assume that $R$ has the R-R diagram in Figure~\ref{positivertype}b. Since $R$ is positive, $Q, T>0$ and $R, U<0$.

\begin{claim}\label{claim2}
In the diagram of $R$ of Figure~\emph{\ref{positivertype}b}, we may assume that $R+Q \neq 0$, or equivalently $(P,Q,R)\neq(2, 1, -1)$.
\end{claim}

\begin{proof}
Suppose $R+Q=0$. Since gcd$(|Q|, |R|)=1$, $(P,Q,R)=(2, 1, -1)$. Then the Heegaard diagram of $R$ has a cut vertex, and $R$ consists of the three types
of two-syllable subwords $AB^S, AB^T$, and $AB^{-U}$ with $|AB^S|=2b, |AB^T|=a$, $|AB^{-U}|=c$. Here $|AB^S|$, for instance, denotes the total number of appearances
of $AB^S$ in $R$ in $\pi_1(H)$. It follows that $|AB|=a+2b+c$ in $R$. Furthermore $\alpha=A^2B^S$ in $\pi_1(H)$, and $\alpha$ and $R$ have no common connections in the $A$-handle.

As in the proof of Claim~\ref{claim1}, since the Heegaard diagram of $R$ has a cut vertex, we perform a change of cutting disks that induces the automorphism of $\pi_1(H)$ taking $A \mapsto AB^{-S}$, and then an orientation-preserving homeomorphism of $H$ inducing
the automorphism $(A,B)\mapsto (B,A^{-1})$ of $\pi_1(H)$. Then $\alpha$ is carried to $A^SB^{2}$ in $\pi_1(H)$, which implies that the R-R diagram of $\alpha$ is also of a rectangular form. The two-syllables $AB^S, AB^T$, and $AB^{-U}$ of $R$ are sent to $B$, $A^{-U}B$, and $A^{T}B$ respectively, which implies that there are only two exponents $-U$ and $T$ with base $A$ in $R$. Thus there are three bands of connections with the label set $(S, T, U)$ in the $A$-handle in the resulting R-R diagram of $\alpha$ and $R$ such that $\alpha$ and $R$ have no common connections in the $A$-handle.
Also we can see that $|AB|$ is reduced strictly to $a+c$ in $R$.

Now the resulting R-R diagram of $\alpha$ and $R$ depends on the determination of the $B$-handle. However, since we have already proved Proposition~\ref{rectangular prop} for all other types of the R-R diagrams of $R$ when $\alpha$ is of a rectangular form, we may assume that it has an R-R diagram of the form in Figure~\ref{positivertype}b with the three labels $(S, T, U)$ in the $A$-handle. If $(S,T,U)\neq(2,1,-1)$ or equivalently $T+U\neq 0$, then the R-R diagram of $R$ satisfies the conclusion of this claim as desired.
If $(S,T,U)=(2,1,-1)$, then we continue to do the process above, which must eventually terminate since $|AB|$ in $R$ is strictly decreasing.
\end{proof}

Now $R+Q \neq 0$ and thus $R$ is connected and has no cut vertex. We apply the similar argument as in the case of R-R diagram of $R$ in Figure~\ref{positivertype}a.
Figure~\ref{horizontalwave2} shows $\omega_h$ when $b>0$. In both of the R-R diagrams $\omega_h$ intersects $\alpha$ transversely once. Similarly when $b=0$, we can show that
$\omega_h$ intersects $\alpha$ transversely once.

\begin{figure}[t]
\includegraphics[width = 1\textwidth]{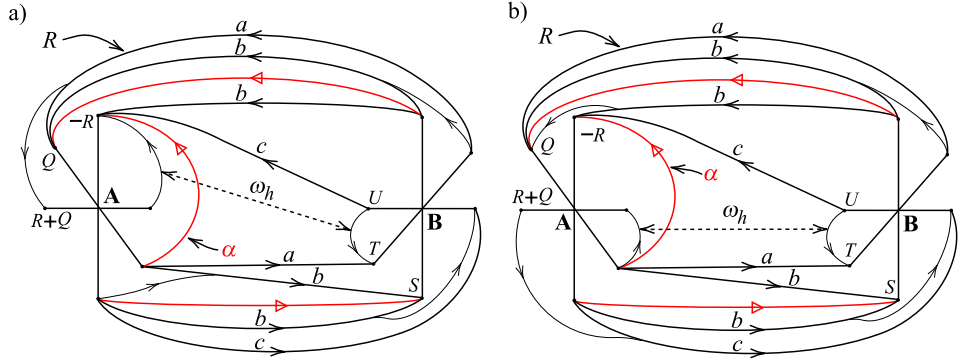}\caption{Horizontal waves $\omega_h$ in the R-R diagram of $R$ when $R+Q>0$ ($R+Q<0$, resp.) in Figure~\ref{horizontalwave1}a (\ref{horizontalwave1}b, resp.) when $b>0$.}\label{horizontalwave2}
\end{figure}

Thus, we have completed the proof of Proposition~\ref{rectangular prop} and therefore Theorem~\ref{main result} when $\alpha$ is Seifert-d and is of a rectangular form.
\end{proof}

\section{Cases in which $\alpha$ is Seifert-$\mathrm{d}$ and has non-rectangular form}
\label{Seifert-d non-rectangular curves}

The goal of this section is to prove Proposition \ref{non-rectangular prop} which shows that Theorem ~\ref{main result} holds in all cases in which $\alpha$ is Seifert-d and has non-rectangular form, i.e., \negthickspace those cases in which $\alpha$ has an R-R diagram with the form of Figure~\ref{PSFFig2a1}b.

\begin{prop}
Theorem~\emph{\ref{main result}} holds if $\alpha$ is Seifert-d and has non-rectangular form.
\label{non-rectangular prop}
\end{prop}

\begin{figure}[t]
\includegraphics[width = 0.55\textwidth]{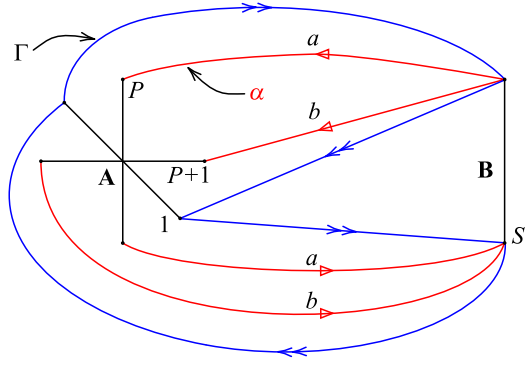}
\caption{An R-R diagram obtained from Figure \ref{PSFFig2a1}b by adding a separating simple closed curve $\Gamma$, disjoint from $\alpha$, to Figure \ref{PSFFig2a1}b so that $\Gamma$ represents $AB^SA^{-1}B^{-S}$ in $\pi_1(H)$. Here $P,S > 1$ with $a,b > 0$, and $\gcd(a,b) = 1$.}
\label{alphagamma}
\end{figure}

\begin{proof}
Note it is possible to add a separating simple closed curve $\Gamma$ to the R-R diagram of $\alpha$ in Figure~\ref{PSFFig2a1}b so that the resulting R-R diagram of $\alpha$ and $\Gamma$ has the form of Figure \ref{alphagamma}. Then $\Gamma$ represents $AB^SA^{-1}B^{-S}$ in $\pi_1(H)$, and $\Gamma$ separates $\partial H$ into two once-punctured tori $F$ and $F'$ with $\alpha \subset F$.

\begin{claim}\label{claim3}
The curve $R$ in $\partial H$ has essential intersections with $\Gamma$.
\end{claim}

\begin{proof}[Proof of Claim \emph{\ref{claim3}}]
Suppose $R$ has no essential intersections with $\Gamma$. Then $R$ lies completely in $F$ or completely   in $F'$.

If $R$ lies completely in $F$, then $\alpha$ and $R$ are isotopic in $\partial H$, but this is impossible since $H[R]$ is hyperbolic, while $H[\alpha]$ is Seifert-fibered.

On the other hand, suppose $R$ lies completely in $F'$. If $R$ has no connections in the $A$-handle, then $R=B^S$ in $\pi_1(H)$, which is a contradiction to that $H[R]$ embeds as a knot exterior in $S^3$ and thus $H_1(H[R])$ is torsion-free. It follows that $R$ has a connection in the $A$-handle and Figure~\ref{alphagamma} implies that $R$ has only one band of connections labeled by 1 in the $A$-handle. If $R$ has a $S$-connection in the $B$-handle, then the Heegaard diagram of $R$ is nonpositive, is connected and has no cut vertex. So there exists a distinguished vertical wave $\omega_v$
yielding a meridian of $H[R]$. It is easy to see from the R-R diagram of $\alpha$ that $\omega_v$ does not intersect $\alpha$, a contradiction. Therefore, $R$ has no $S$-connections and thus at most two bands of connections in the $B$-handle, implying by Proposition~\ref{nonhyperbolic R} that $H[R]$ is nonhyperbolic, a contradiction.

It follows $R$ has essential intersections with $\Gamma$.
\end{proof}

Next, consider Figure \ref{gamma1} which shows $F$ cut open along two properly embedded arcs in $F$ parallel to $\partial D_A \cap F$ and  $\partial D_B \cap F$. Note that since $\partial D_A \cap F$ is a single connection in $F$, and $|\alpha \cap \partial D_A|=(a+b)P+b$, one has $c=(a+b)(P-1)+b$. And therefore, since $a+b \geq 2$, and $P > 1$, one has $c > a+b > 2$.

\begin{figure}[t]
\includegraphics{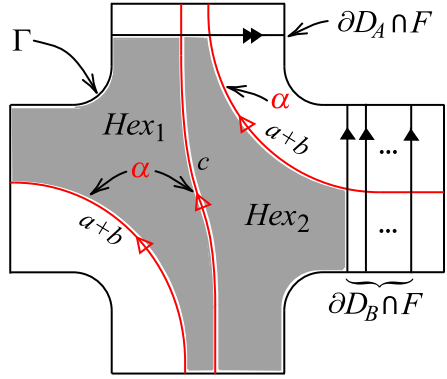}\caption{Configuration of $\alpha, \partial D_A$ and $\partial D_B$ in $F$.}\label{gamma1}
\end{figure}

The simple closed curve $\alpha$ together with the arcs of $\partial D_B \cap F$ and the arc of $\partial D_A \cap F$ cut $F$ into a number of faces, each of which is a rectangle, except for the pair of hexagonal faces $Hex_1$ and $Hex_2$, shown as shaded regions in Figure~\ref{gamma1}. Now it is easy to see that any connection in $F$ disjoint from $\alpha$ traverses each of the above rectangles. Since $R$ is disjoint from $\alpha$, and $R$ has essential intersections with $F$, $R \cap F$ contains such connections. Then, because $a+b \geq 2$ and $c > a+b$, we see that $A^m$ and $B^n$ appear in the cyclic word which $R$ represents in $\pi_1(H)$ with $|m|, |n| >1$. It follows that the Heegaard diagram $\mathcal{D}$ of $R$ with respect to $\{\partial D_A, \partial D_B\}$ is connected and has no cut vertex. Therefore the invariant arc $\omega$ promised by \cite{B20} appears in $\mathcal{D}$ as a distinguished wave based at $R$.

\begin{claim}\label{walpha1}
$\omega$ intersects $\alpha$ transversely in one point.
\end{claim}

\begin{proof}[Proof of Claim \emph{\ref{walpha1}}]
If $\omega$ is disjoint from $\alpha$, then $\alpha$ is isotopic to a meridian of $H[R]$, a contradiction. Therefore $\omega$ intersects $\alpha$.

Suppose $p$ is a point of $\omega \cap \alpha$. Then $p$ lies in the boundary of a rectangular face, say $R_p$, of $F$. $R_p$ is
traversed by at least one connection of $R\cap F$ which we may assume has the same orientation as $\alpha$. But, since $p$ is
essential, one of the endpoints of $\omega$, say $p'$ must also lie in $R_p$ on a subarc of a connection of $R\cap F$. So
we have the configuration shown in Figure~\ref{gamma2}.

\begin{figure}[t]
\includegraphics[width = 0.8\textwidth]{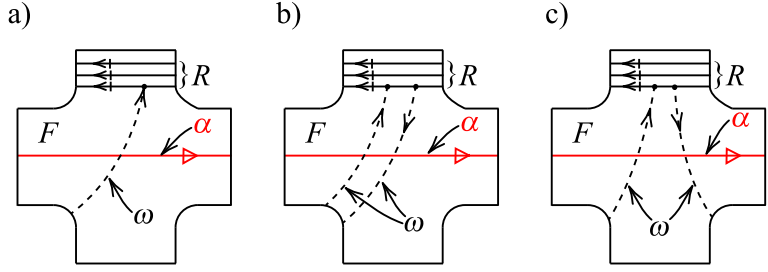}
\caption{The possible configurations of $\alpha$, $R$ and $\omega$ in $F$.}
\label{gamma2}
\end{figure}

Since $\omega$ is an arc and has only two endpoints, it follows that $\omega \cap \alpha$ consists of either one or two points, and if $\omega \cap \alpha$ consists of two points, then these two intersections have opposite signs because of the definition of a wave. However this
is impossible, because if $\omega \cap \alpha$ consists of two points of intersection with opposite signs, then the algebraic intersection number of $\omega$ and $\alpha$ is
equal to $0$, which implies that the geometric intersection number of a meridian representative $M$ and $\alpha$ is equal to $0$ and thus
$\alpha$ is a meridian of $H[R]$, a contradiction.
\end{proof}

Thus, we have completed the proof of Proposition~\ref{non-rectangular prop} and therefore Theorem~\ref{main result} when $\alpha$ is Seifert-d and is of a non-rectangular form.
\end{proof}

\section{The case when $\alpha$ is Seifert-$\mathrm{m}$}
\label{Seifert-m curves}

In this section, we prove Theorem~\ref{main result} when $\alpha$ is Seifert-m in a genus two handlebody $H$.
It follows from the classification theorem of Seifert-m curves in \cite{K20b} that $\alpha$ has an R-R diagram of the form in Figure~\ref{Seifert-m curves 1} with $S>1$.

\begin{prop}
Theorem~\emph{\ref{main result}} holds if $\alpha$ is Seifert-m.
\label{Seifert-m prop}
\end{prop}

\begin{figure}[t]
\includegraphics[width = 0.55\textwidth]{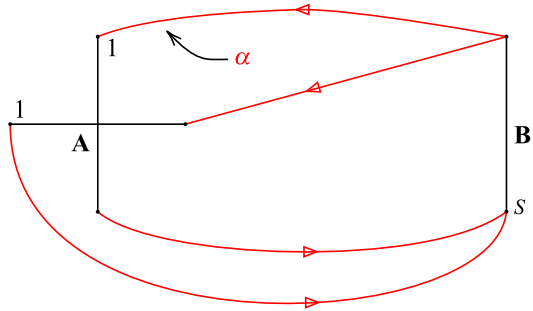}
\caption{An R-R diagram of a Seifert-m curve $\alpha$ on $\partial H$, where $S>1$.}
\label{Seifert-m curves 1}
\end{figure}

\begin{figure}[t]
\includegraphics[width = 0.6\textwidth]{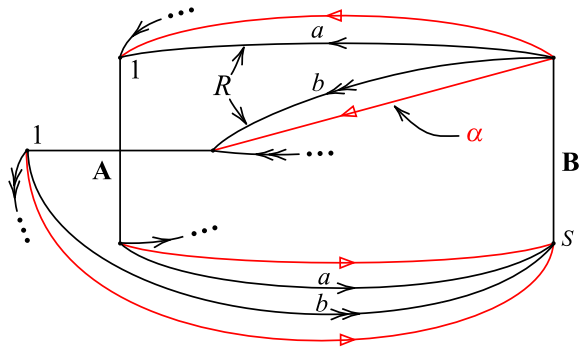}
\caption{If an R-R diagram, say $\mathcal{D}$, of a Seifert-m curve $\alpha$ and disjoint curve $R$ contains a subdiagram with the form of this figure with both $a > 0$ and $b > 0$, then the corresponding Heegaard diagram of $\mathcal{D}$ is connected, has no cut-vertices, and there is a vertical wave $\omega_v$ based at $R$ in $\mathcal{D}$ such that $\omega_v$ is disjoint from $\alpha$. (Note this is true even if the orientations of the $b$-weighted bands in this figure are reversed.)}
\label{Seifert-m curves 2}
\end{figure}

\begin{proof}

We observe from Figure~\ref{Seifert-m curves 1} that $\alpha$ has two bands of connections labelled by $1$ in the $A$-handle.

\begin{claim}\label{claim4}
$R$ must have only one band of connections labelled by $1$ in the $A$-handle.
\end{claim}

\begin{proof}
Suppose for a contradiction that $R$ has no $1$-connections in the $A$-handle or
$R$ has the two bands of connections labelled by $1$ in the $A$-handle.

First, suppose $R$ has no $1$-connections in the $A$-handle.
If $R$ has either no connections or only $0$-connections in the $A$-handle,
then $R$ should have only one $S$-connection in the $B$-handle, which implies that $R=B^S$ in $\pi_1(H)$.
This is impossible since $H[R]$ embeds as a knot exterior in $S^3$ and thus $H_1(H[R])$ is torsion-free.
Thus $R$ has only $2$-connections in the $A$-handle. If $R$ has a $S$-connection in the $B$-handle, then it is easy to see
that the Heegaard diagram of $R$ is nonpositive, connected and has no cut-vertex. Thus there exists a distinguished vertical wave $\omega_v$ yielding a meridian of $H[R]$.
It follows from the R-R diagram of $\alpha$ that $\omega_v$ does not intersect $\alpha$, which is a
contradiction. Now $R$ has only two bands of connections in the $B$-handle. However, this also cannot happen by Proposition~\ref{nonhyperbolic R}
indicating that $H[R]$ is not hyperbolic.

Now we suppose that $R$ has the two bands of connections labelled by $1$ in the $A$-handle.
Orient $R$ so that the labels at the ends of the two bands where $R$ enters are either both 1, or 1 and $-1$.
If the two labels are 1 and $-1$, then an R-R diagram of $\alpha$ and $R$ contains
a subdiagram with the form of Figure~\ref{Seifert-m curves 2} with both $a>0$ and $b>0$. Thus Heegaard diagram of $R$ is nonpositive and also is connected and has no cut-vertex.  It follows that a distinguished vertical wave $\omega_v$ yielding a meridian of $H[R]$ does not intersect $\alpha$, a contradiction.

If the two labels are both 1, then it follows from the R-R diagram of $\alpha$ that
$R$ must have both $S$- and $(-S)$-connections. Note that in this case an R-R diagram of $\alpha$ and $R$ also contains a subdiagram with the form of Figure~\ref{Seifert-m curves 2} with both $a>0$ and $b>0$ and with orientations of the b-weighted bands reversed.
This implies that $R$ is nonpositive,
is connected and has no cut-vertex. By the similar argument above, a distinguished vertical wave $\omega_v$ yielding a meridian of $H[R]$ does not intersect $\alpha$, a contradiction.
\end{proof}

\begin{figure}[t]
\includegraphics[width = 0.65\textwidth]{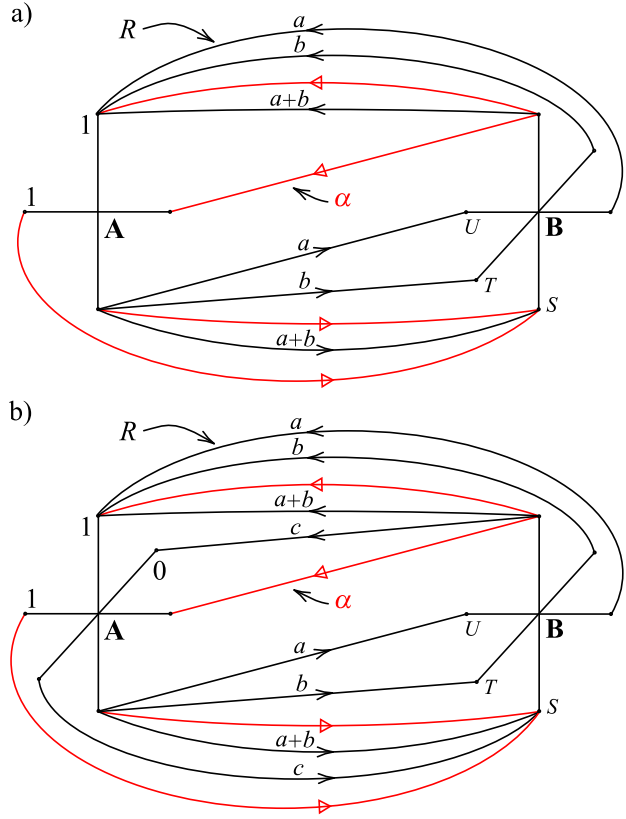}
\caption{The two R-R diagrams of disjoint curves $\alpha$ and $R$ in which $\alpha$ is Seifert-m and $R$ has no $2$-connections on the A-handle.}
\label{no 2-connections}
\end{figure}

By Claim~\ref{claim4}, $R$ has only one band of connections with label $1$ in the $A$-handle. There are two bands of connections with label $1$ in
the $A$-handle in the R-R diagram of $\alpha$: say, vertical and horizontal. Applying an orientation-reversing homeomorphism of $H$, if necessary, we may assume
without loss of generality that $R$ has vertical $1$-connections in the $A$-handle.
Now we break the argument into two cases: (1) $R$ has no $2$-connections and (2) $R$ has $2$-connections in the $A$-handle.

\smallskip
\textbf{Case (1): R has no 2-connections in the A-handle.}
\smallskip

There are two possible R-R diagrams of $R$ as shown in Figure~\ref{no 2-connections} depending on whether or not $R$ has $0$-connections. Note that $a,b>0$ in the R-R diagram of Figure~\ref{no 2-connections}a and $a+b>0$
in the R-R diagram of Figure~\ref{no 2-connections}b, because otherwise $H[R]$ is not hyperbolic.

If $R$ in Figure~\ref{no 2-connections} is nonpositive, then it is easy to see from the R-R diagrams that a distinguished vertical wave $\omega_v$
yielding a meridian of $H[R]$ intersects $\alpha$ transversely at a point.

If $R$ in Figure~\ref{no 2-connections} is positive, then the Heegaard diagram of $R$ has a cut-vertex.
Since $a+b>0$ in the R-R diagram of Figure~\ref{no 2-connections}b, either $a>0$ or $b>0$.
Without loss of generality we may assume that $b>0$. Therefore $b>0$ in both of the R-R diagrams in Figure~\ref{no 2-connections}, which implies $R$ has a subword $\cdots B^SAB^TAB^S\cdots$. As we did in Claim~\ref{claim1} in Section~\ref{Seifert-d rectangular curves}, we perform a change of cutting disks of $H$ inducing
an automorphism of $\pi_1(H)$ taking $A\mapsto AB^{-T}$. Using a hybrid diagram
we can see that since $\alpha=AB^SA^{-1}B^S$ in $\pi_1(H)$,
under this automorphism $\alpha$ remains same, i.e., $\alpha$ has the same form of R-R diagram in Figure~\ref{Seifert-m curves 1} while since the subword $\cdots B^SAB^TAB^S\cdots$ of $R$ is sent to $\cdots B^{-U}A^2B^{-U}\cdots$, $R$ is transformed into a simple closed curve whose word in $\pi_1(H)$ contains $A^2$. This implies that $R$ has $2$-connections in the $A$-handle. So this case
belongs to Case (2) where $R$ has $2$-connections in the $A$-handle, which is handled next.

\begin{figure}[t]
\includegraphics[width =0.65\textwidth]{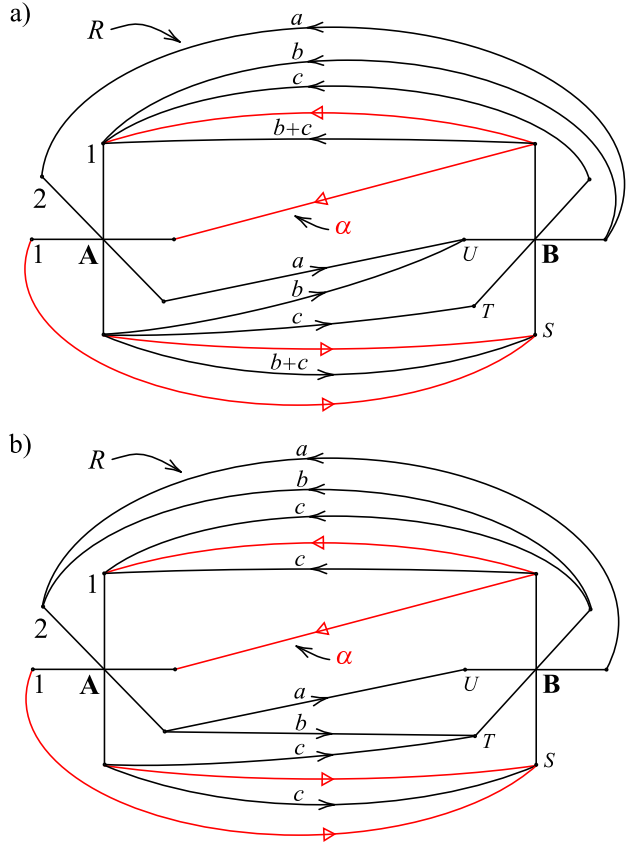}
\caption{The two R-R diagrams of disjoint curves $\alpha$ and $R$ in which $\alpha$ is Seifert-m and $R$ has $2$-connections on the A-handle.}
\label{with 2-connections}
\end{figure}

\smallskip
\textbf{Case (2): R has 2-connections in the A-handle.}
\smallskip

There are two possible R-R diagrams of $R$ as illustrated in Figure~\ref{with 2-connections}. Note that $a, b>0$ in Figure~\ref{with 2-connections}a and $b, c>0$ in Figure~\ref{with 2-connections}b. This is because for the R-R diagram of $R$ in Figure~\ref{with 2-connections}a, since $R$ has 2-connections in the $A$-handle, $a>0$. If $b=0$ there,
then since $R$ is a simple closed curve, $c=0$, which implies by Proposition~\ref{nonhyperbolic R} that $H[R]$ is not hyperbolic. For the R-R diagram of $R$ in Figure~\ref{with 2-connections}b, if $c=0$, then Proposition~\ref{nonhyperbolic R} implies that $H[R]$ would not be hyperbolic. If $b=0$ there, then $a=0$ and thus $H[R]$ is not hyperbolic.

\begin{figure}[t]
\includegraphics[width = 0.65\textwidth]{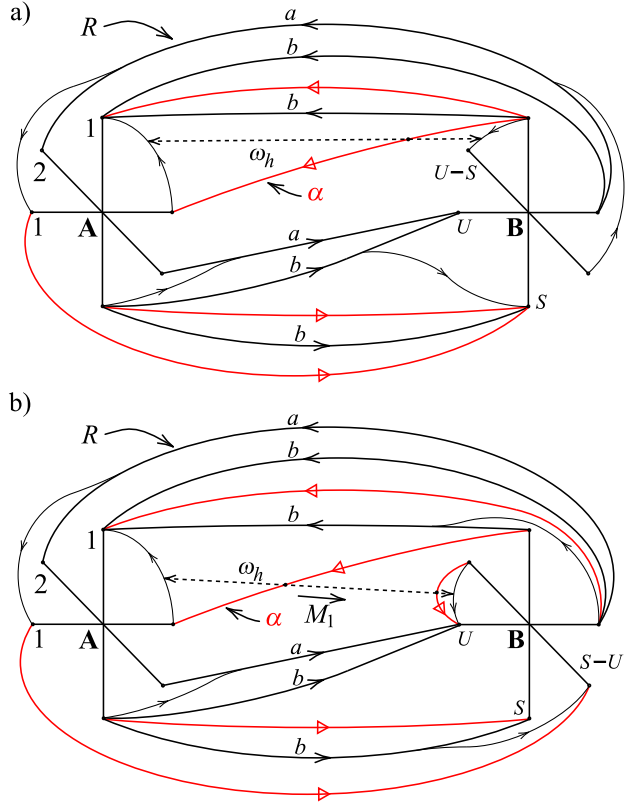}
\caption{Horizontal waves $\omega_h$ in R-R diagrams of $\alpha$ and $R$ when $U>S$ in Figure~\ref{positivewavesMobius}a, and $U<S$ in Figure~\ref{positivewavesMobius}b.}
\label{positivewavesMobius}
\end{figure}

If $R$ in Figure~\ref{with 2-connections} is nonpositive, as in the case (1), there exists a distinguished vertical wave $\omega_v$ yielding a meridian of $H[R]$ which intersects $\alpha$ transversely once.

We assume that $R$ in Figure~\ref{with 2-connections} is positive. From the conditions that $a, b>0$ in Figure~\ref{with 2-connections}a and $b, c>0$ in Figure~\ref{with 2-connections}b, it follows that the Heegaard diagrams of $R$ are connected and has no cut-vertex. Therefore there exists a distinguished horizontal wave $\omega_h$ yielding a meridian of $H[R]$.

If $c>0$($a>0$, resp.) in Figure~\ref{with 2-connections}a(\ref{with 2-connections}b, resp.), then $R$ has all of the three bands of connections of labels $U, T, S$ in the $B$-handle. Since $R$ is positive, all of $U,T,$ and $S$ are positive and thus $T$ is the maximal label of connections in the $B$-handle. Therefore, as in Figure~\ref{positiverwave} or in Figure~\ref{horizontalwave1} a horizontal wave $\omega_h$ can be located in the R-R diagram of $R$
by isotoping the $2$-connection and the $T$-connection in the $A$- and $B$-handle respectively. Then we can see that $\omega_h$ intersects $\alpha$ once.

If $c=0$($a=0$, resp.) in Figure~\ref{with 2-connections}a(\ref{with 2-connections}b, resp.),
then the two R-R diagrams in Figure~\ref{with 2-connections} have the same form. In other words,
the R-R diagram of Figure~\ref{with 2-connections}b with $a=0$ is the R-R diagram of Figure~\ref{with 2-connections}a with $c=0$ by replacing $(b,c,T)$ by $(a,b,U)$.
Therefore we focus only on the R-R diagram of $R$ in Figure~\ref{with 2-connections}a with $c=0$.
Locating a horizontal wave $\omega_h$ in the R-R diagram of $R$ depends on the sizes of $U$ and $S$ in the $B$-handle. Figure~\ref{positivewavesMobius}a(\ref{positivewavesMobius}b, resp.) shows $\omega_h$ when $U>S$($U<S$, resp.). It follows that when $U>S$, $\omega_h$ intersects $\alpha$ at a point.
On the other hand, when $U<S$, $\omega_h$ intersects $\alpha$ twice as shown in Figure~\ref{positivewavesMobius}b, where one $S$-connection of $\alpha$ is isotoped. However,
it is easy to see from the R-R diagram that one meridian representative $M_1$ obtained from $R$ by surgery along $\omega_h$ represents $AB^UAB^U$ in $\pi_1(H)$.
This is impossible because $H[M_1]$ also embeds in $S^3$ as a knot exterior and thus $H_1(H[M_1])$ is torsion-free.

Thus, we have completed the proof of Proposition~\ref{Seifert-m prop} and therefore Theorem~\ref{main result} when $\alpha$ is Seifert-m.
\end{proof}



\begin{thebibliography}{XXXXX}


\bibitem[B20]{B20}
	Berge, J.,
	\emph{Distinguished waves and slopes in genus two},
    preprint.

\bibitem[B93]{B93}
	Berge, J.,
	\emph{Embedding the Exteriors of One-Tunnel Knots and Links in the 3-Sphere},
	Unpublished transparencies of invited address at Cascade Topology Conf. Spring 1993.	
\bibitem[BZ98]{BZ98}
	Boyer, S., Zhang, X.,
	\emph{On Culler-Shalen seminorms and Dehn filling},
	Ann. of Math. \textbf{148} (1998), 737--801.

\bibitem[CGLS87]{CGLS87}
	Culler, M., Gordon, C. McA., Luecke, J., Shalen, P.B.,
	\emph{Dehn surgery on knots},
	Ann. of Math. \textbf{125} (1987), 237--300.

\bibitem[D03]{D03}
	Dean, J.,
	\emph{Small Seifert-fibered Dehn surgery on hyperbolic knots},
	Algebraic and Geometric Topology \textbf{3} (2003), 435--472.	

\bibitem[DMM12]{DMM12}
	Deruelle, A., Miyazaki, K., and Motegi, K.,
	\emph{Networking Seifert Surgereis on Knots},
	Mem. Amer. Math. Soc. \textbf{217} (2012) viii+130.

\bibitem[DMM14]{DMM14}
	Deruelle, A., Miyazaki, K., and Motegi, K.,
	\emph{Neighbors of Seifert surgeries on a trefoil knot in the Seifert Surgery Network},
	Boletín de la Sociedad Matemática Mexicana \textbf{20} no. 2, (2014) 523--558.

\bibitem[EJMM15]{EJMM15}	
	Eudave-Mu\~noz, M., Jasso, E., Miyazaki, K. and Motegi, K.,
	\emph{Seifert fibered surgereis on strongly invertible knots without primitive/Seifert positions},
	Top. and its Appli. \textbf{196} Part B, (2015),  729--753.

\bibitem[EM92]{EM92}	
	Eudave-Mu\~noz, M.,
	\emph{Band sums of links which yield composite links. The cabling conjecture for strongly invertible knots},
	Trans. Amer. Math. Soc. \textbf{330} no. 2, (1992),  463--501.

\bibitem[GL95]{GL95}	
	Gordon, C. McA., Luecke, J.,
	\emph{Dehn surgeries on knots creating essential tori, I},
	Comm. Anal. Geom. \textbf{3} (1995), 597--644.

\bibitem[HOT80]{HOT80}
	Homma, T., Ochiai, M. and Takahashi, M.,
	\emph{An Algorithm for Recognizing $S^3$ in 3-Manifolds with Heegaard Splittings of Genus Two},
	Osaka J. Math. \textbf{17} (1980), 625--648.

\bibitem[K20a]{K20a}
   Kang, S. \emph{On nonhyperbolicity of P/P and P/SF knots in $S^3$},
    preprint.

\bibitem[K20b]{K20b}
   Kang, S. \emph{Primitive, proper power, and Seifert curves in the boundary of a genus two handlebody},
    preprint.

\bibitem[K20c]{K20c}
   Kang, S. \emph{Tunnel-number-one knot exteriors in $S^3$ disjoint from proper power curves},
    preprint.

\bibitem[MMM05]{MMM05}
   Mattman, T., Miyazaki, K. and Motegi, K.,
	\emph{Seifert-fibered surgeries which do not arise from primitive/Seifert-fibered constructions},
	Trans. Amer. Math. Soc. \textbf{358} no. 9, (2005),  4045--4055.

\bibitem[O79]{O79}
	Ochiai, M.,
	\emph{Heegaard-Diagrams and Whitehead-Graphs},
	Math. Sem. Notes of Kobe Univ. \textbf{7} (1979), 573--590.	

\bibitem[T07]{T07}
	Teragaito, M.,
	\emph{A Seifert fibered manifold with infinitely many knot-surgery descriptions},
	Int. Math. Res. Not. \textbf{9} (2007), Art. ID rnm 028, 16 pp.

		
\end{thebibliography}
\end{document}